\documentclass[12pt]{amsart}
 \usepackage{graphicx}
\usepackage{epsfig}
\usepackage{amsfonts}
\usepackage{amscd}
\usepackage{amsfonts}
\usepackage{color}
\usepackage{psfrag}
\usepackage{pstricks}
\usepackage{mathdots}
\usepackage{psfrag}
\usepackage{amsthm,amsmath,amssymb,graphics,hyperref}
\usepackage{amsmath,amssymb,hyperref} 
\usepackage{enumerate}
\usepackage[titletoc,title]{appendix}
\input xy
\xyoption{all}
\usepackage{ulem}

\theoremstyle{plain}
\newtheorem{theorem}{Theorem}[section]
\newtheorem{proposition}[theorem]{Proposition}
\newtheorem{lemma}[theorem]{Lemma}
\newtheorem{corollary}[theorem]{Corollary}
\theoremstyle{definition}

\newtheorem{remark}[theorem]{Remark}

\newcommand{\R}{\mathbb{R}}

\newcommand{\ra}{{\rightarrow}}

\let\cal\mathcal
\newcommand{\bR}{{\mathbb R}}

\def\P{\mathbb{P}}

\begin{document}
\title [SRB Entropy of  Convex Projective Structure]
      { Continuity of the SRB Entropy of  Convex Projective Structures}
        \author{Patrick Foulon and Inkang Kim}

        \date{}
        \maketitle
\begin{abstract}
The space of convex projective structures has been well studied with respect to the topological entropy. But, to better understand the geometry of the structure, we study the entropy of the Sinai-Ruelle-Bowen measure and show that it is a continuous function.
\end{abstract}
 \footnotetext[3]{I. Kim
gratefully acknowledges the partial support of Grant
(NRF-2017R1A2A2A05001002) and the warm support of CIRM during his
visit.}
\vskip 0.1 in
{\centerline {\it in memoriam Rufus Bowen}}

\section{Introduction}
A topological manifold $M$ can be equipped with a $(G,X)$-structure where $X$ is a model space and $G$ is a group acting on $X$, so that $M$ has an atlas $\{(\phi_i, U_i)\}$ from
open sets $\{U_i\}$ in $M$ into open sets in $X$ and the transition maps $\{\phi_i\circ \phi_j^{-1}\}$ are restrictions of  elements in $G$.
Depending on the choice of $(G,X)$, many interesting geometric structures can arise.  For instance, if $M$ is a closed surface with  genus at least 2, a hyperbolic
structure corresponds to $(\text{PSL}(2,\mathbb R),\mathbb H^2)$, a real projective structure  to $(\text{PGL}(3,\mathbb R), \mathbb{RP}^2)$, and a complex projective structure to $(\text{PSL}(2,\mathbb C), \mathbb{CP}^1)$.

Given a geometric structure, there exist a developing map $D:\widetilde M \ra X$ which is a local homeomorphism and a holonomy representation $\rho:\pi_1(M)\ra G$, so that $D$ is $\rho$-equivariant. In this paper, we study the real projective structures $(\text{PGL}(n+1,\mathbb R), \mathbb{RP}^n)$, and especially the strictly convex real projective structures $\cal P_n(M)$ on a closed manifold $M$,  that is when  $\Omega = D(\widetilde M)$ is a strictly convex open domain of $\mathbb{RP}^n$. 

The set $\cal P_n(M)$ of equivalence classes of such strictly convex real projective structures, in the case of a closed surface, is known as Hitchin component in a character variety \cite{cg, Hi}. It has drawn much attention recently and many aspects of the set have been studied. The space has been identified as the holomorphic vector bundle over Teichm\"uller space \cite{BH, Cal, CY, Lab,Lo2} using affine sphere theory. It is also identified with the space of Anosov representations \cite{La}. Explicit coordinates are studied \cite{bg, bk, fg, Goldman} and the space has a mapping class group invariant K\"ahler metric \cite{KZ}.
In this paper, we want to address the dynamical aspects related to the geodesic flow on the tangent bundle in any dimension.
Crampon  developed many aspects of this point of view in his papers \cite{Crampon, crampon, cra, cra1}.
There are several measures invariant under the geodesic flow, but we single out one measure, the Sinai-Ruelle-Bowen measure, abbreviated SRB measure. That measure provides, especially in dimension two, some  deep insight concerning the shape of the boundary at infinity. We will elaborate on this in the text.

The most studied invariant measure is the Bowen-Margulis measure. A key property used by Crampon is that this measure is reversible, i.e., it is invariant under the
flip map $\sigma(x, [v])=(x,[-v])$. The most eminent consequence is that the sum of the positive Lyapunov exponents associated to that measure is $n-1$, as in the hyperbolic metric case. 

Using the Ruelle inequality for the Bowen-Margulis measure, 
Crampon deduces that the topological entropy is less than or equal to $n-1$ with equality only in the hyperbolic case.   One key feature of the SRB measure is that the Ruelle inequality becomes an equality, known as the Pesin formula. When there is an invariant volume form, that form gives, after normalization, the density of the SRB measure. But we know from  Benoist \cite{Be} that if the convex structure is not the hyperbolic model, then there is no such an invariant volume form for the flow.
Nevertheless, in general the SRB measure is characterised by the fact that its conditional measures on unstable leaves are absolutely continuous to the Lebesgue measure (Ledrappier-Young  \cite{LY}).  This is why it turns out to be a key ingredient to better understand the geometry of projective structures. We also observe a global irreversible effect along the orientation of time. We will elaborate on this in Corollary \ref{irreversible} (3).
\\
\begin{theorem}\label{main} Let $M$ be a closed manifold of dimension $n$. Then the map $h_{SRB}:\cal P_n(M)\ra \mathbb R$ is continuous, where $h_{SRB}$ denotes the Sinai-Ruelle-Bowen measure entropy of the geodesic flow defined on the convex real projective manifold.
\end{theorem}

Given a strictly convex real projective manifold $M$ and the geodesic flow invariant SRB measure $\mu_{SRB}$, by ergodicity of the SRB measure,
there exists a set $W_{SRB}\subset HM=(TM\setminus \{0\})/\mathbb R^*_+$ with $\mu_{SRB}(W_{SRB})=1$, and $\chi_1,\cdots,\chi_p \in \mathbb R$ called the Lyapunov exponent relative to the SRB measure,
such that, for any $w\in W_{SRB}$, there exists a geodesic flow invariant decomposition $THM=\oplus E_i$ and for any $v\in E_i$,  $\chi_i(w,v)=\chi_i$, where $\chi_i(w,v)$ are Lyapunov exponents. See Sections \ref{pre} and \ref{Ly} for definitions and properties.
\begin{corollary}\label{irreversible}Let $M_t$ be a smooth family of real projective structures on $M$. Then
\begin{enumerate}
\item The  sum $\chi^+$ of the SRB positive Lyapunov exponents vary continuously in $t$.

\item  The entropy $h_{SRB}$ of the SRB measure satisfies  inequality
$$h_{SRB}=n-1 + \frac{1}{2}\eta \leq n-1$$   where $\eta$ is the SRB almost sure value of $\eta(w)=\sum \eta_i(w) {\text dim}E_i(w)$, the sum of parallel Lyapunov exponents. By Ruelle inequality, $\eta\leq 0$ and the equality holds if and only if the structure is hyperbolic.

\item (Irreversibility) Let $dvol$ be any continuous volume form. Then
$$\lim_{s\ra\infty} \frac{1}{s}\log(\frac{d(\phi^s)^*(dvol)}{dvol})=\sum \eta_i \text{dim} E_i= \eta.$$
\end{enumerate}
\end{corollary}

As shown by Crampon \cite{cra1}, the Lyapunov exponents have to do with the convexity of the boundary of $\Omega$. It is particularly meaningful when the dimension of $\Omega$ is 2. The boundary $\partial \Omega$ in the neighborood of a point $p$ can be written as the graph of a convex function $f$ near the origin. Such a function $f$ is said to be {\it approximately $\alpha$-regular} at $p$ for $\alpha\in [1,\infty]$, if
$$\lim_{t\ra 0} \frac{\log \frac{f(t)+f(-t)}{2}}{\log |t|}=\alpha (p).$$ 
This quantity is invariant under affine and projective transformations. Approximately $\alpha$-regular means if $\alpha < \infty$ that the function behaves like $|t|^\alpha$ near the origin. 

\begin{corollary}\label{shape}
There is a set $F_{SRB} \subset \partial \Omega$ of full Lebesgue measure and a real number, $ 2 \leq \alpha_{SRB} < \infty$, called the regular exponent for the boundary, such that for any $p\in F_{SRB} $ we have $$\alpha(p) =\alpha_{SRB}$$
The map $\alpha_{SRB}:\cal P_2(M)\ra \mathbb R$ is continuous.\\
Furthermore, we have $$\chi^+ \alpha=2.$$
and
$\alpha_{SRB} = 2 $ if and only if the structure is hyperbolic.
\end{corollary}

\subsection*{Remarks}
\begin{itemize}
\item This shows that if a stritcly convex projective structure is not the standard hyperbolic structure then the boundary is ``very" flat. In particular for almost every point $p$ on the boundary the local representing graph admits at least a  second derivative at $p$ that vanishes.  But the boundary is nowhere $C^2$. This remark makes more precise  the work of  Benoist \cite{Be} showing that the curvature of the boundary is localized on a set of null Lebesgue measure.
\item  In higher dimensions, similar exponents were introduced by Crampon and the general statement that we obtain is that their harmonic mean varies continuously.
\item Something  worth noticing and a bit surprising is  there exists no ideal regular odd-sided geodesic polygon if $\Omega$ is not conic. See the proof of  Corollary \ref{shape} and remark therein.
 \end{itemize}


{\bf  Acknowledgement}.  We thank  F. Ledrappier for numerous discussions.

\section{preliminaries}\label{pre}
\subsection{Projective structure}
Let $M$ be an $n$-dimensional closed manifold equipped with a strictly convex projective structure so that
$D:\widetilde M\ra \mathbb{RP}^n$ is a developing map with $D(\widetilde M)=\Omega$. Then it inherits 
a Hilbert metric defined in the following way. Let $x,y\in M$. Then choose $\tilde x,\tilde y$ in the same fundamental domain, and define $d(x,y):=d_\Omega(D(\tilde x), D(\tilde y))$ where the Hilbert metric $d_\Omega$ defined
as below and the corresponding Finsler norm on $T\Omega$ is denoted by $F$. More precisely,
 for $x\neq y\in \Omega$, let $p,q$ be the
intersection points of the line $\overline{xy}$ with $\partial\Omega$ such that $p,x,y,q$ are in this order. The
Hilbert distance is defined by
$$d_\Omega(x,y)=\frac{1}{2}\log \frac{|p-y||q-x|}{|p-x||q-y|}$$ where $| \cdot |$ is a Euclidean
norm in an affine chart containing $\Omega$. This metric coincides with the hyperbolic metric if
$\partial\Omega$ is a conic. Note that the Hilbert metric on $M$ also depends on the boundary of $\Omega$. The Hilbert metric is Finsler rather than
Riemannian. The Finsler norm $F=||\cdot||$ is given, for $x\in\Omega$ and a
vector $v$ at $x$, by
$$||v||_x=\big(\frac{1}{|x-p^-|}+\frac{1}{|x-p^+|}\big)|v|$$ where $p^\pm$
are the intersection points of the line with $\partial\Omega$, defined by $x$ and $v$ with the
obvious orientation, and where $| \cdot |$ is again 
a Euclidean norm.

If $\Omega$ admits a compact quotient manifold $M=\Omega/\Gamma$, then $\partial \Omega$ is $C^{1+\alpha}$, $\Gamma$ is Gromov hyperbolic and the geometry behaves like a negatively curved case.  See \cite{Be, Ben}.

For a given $w=(x,[\xi])\in H\Omega=(T\Omega\setminus \{0\})/\R_+^*$, the unstable manifold $W^{su}$ passing through $w$ is defined to be
$$W^{su}(w)=\{(y,[\phi])\in H\Omega| \xi(-\infty)=\phi(-\infty), y\in \cal H_{\sigma w}\}.$$ Here $\xi(-\infty)$ denotes $\gamma_\xi(-\infty)$ where $\gamma_\xi$ is the geodesic determined by $\xi$, and $\sigma w=(x,[-\xi])$ is a flip map, and $\cal H_w$ is the horosphere based at $\xi(\infty)$ passing through $x$.
Similarly one can define a stable manifold
$$W^{ss}(w)=\{(y,[\phi])\in H\Omega| \xi(\infty)=\phi(\infty), y\in\cal H_w\}.$$ These stable and unstable manifolds are $C^1$ if $\partial \Omega$ is $C^1$.

The tangent spaces of $W^{su}$ and $W^{ss}$ form unstable and stable vectors in $TH\Omega$, i.e., along the geodesic flow, they
 expand or decay exponentially.  All the objects defined above descend to the quotient manifold $M$ and it is known \cite{Be} that the geodesic flow on $HM$ is Anosov with invariant decomposition
$$THM=\R X\oplus E^s\oplus E^u,$$ where $X$ is the vector field generating the geodesic flow.

\subsection{Lyapunov exponents and Sinai-Ruelle-Bowen measure}\label{Ly}
\subsubsection{Lyapunov exponent}
Let $\phi=\phi^t$ be a $C^1$ flow on a Riemannian manifold $W$. A point $w\in W$ is said to be regular if there exists
a $\phi^t$-invariant decomposition
$$TW=E_1\oplus \cdots \oplus E_p$$ along $\phi^t w$ and real numbers
$$\chi_1(w)<\cdots< \chi_p(w),$$ such that, for any vector $Z_i\in E_i\setminus \{0\}$,
$$\lim_{t\ra\pm \infty}\frac{1}{t}\log ||d\phi^t(Z_i)||=\chi_i(w),$$ and
\begin{eqnarray}\label{det}
\lim_{t\ra\pm \infty}\frac{1}{t}\log |\text{det} d\phi^t|=\sum_{i=1}^p\text{dim}E_i \cdot \chi_i(w).
\end{eqnarray}
The numbers $\chi_i(w)$ associated with a regular point $w$ are called the Lyapunov exponents of the flow at $w$.
Due to Oseledets'  multiplicative ergodic theorem \cite{Os}, the set of regular points has  full measure.
\begin{theorem}Let $\phi$ be a $C^1$ flow on  a Riemannian manifold $W$ and $\mu$ a $\phi^t$-invariant probability measure. If
$$\frac{d}{dt}|_{t=0} \log ||d\phi^{\pm t}||\in L^1(W,\mu),$$ then the set of regular points has full measure.
\end{theorem}

\subsection{In the case of convex projective structures}
There exists an extended Finsler norm $\bar F$ on $HM=(TM\setminus \{0\})/\bR^*_+$ and the Lyapunov exponent is defined to be
$$\lim_{t\ra\pm \infty}\frac{1}{t}\log \bar F(d\phi^t(Z_i))=\chi_i(w).$$

Let $\Omega$ be a strictly convex domain with $C^1$ boundary equipped with a Hilbert metric. There is a notion of parallel transport $T^t$ due to Foulon \cite{Fu1, Fu2}. The parallel Lyapunov exponent of $v\in T_x\Omega$ along $\phi^t(x,[\psi])$ is defined to be
$$\eta((x,[\psi]),v)=\lim_{t\ra\infty}\frac{1}{t}\log F(T^t(v)).$$

Then Crampon  \cite{cra} showed that a point $w=(x,[\psi])\in H\Omega$ is regular if and only if there exists a decomposition
$$T_x\Omega=\R \psi\oplus E_0(w)\oplus (\oplus_{i=1}^p E_i(w))\oplus E_{p+1}(w),$$ and real numbers
$$-1=\eta_0(w)<\eta_1(w)<\cdots <\eta_p(w)< \eta_{p+1} = 1,$$
such that
for any $v_i\in E_i\setminus \{0\}$,
$$\lim_{t\ra\pm\infty}\frac{1}{t}\log F(T^t_w(v_i))=\eta_i(w),$$ and
$$\lim_{t\ra\pm\infty}\frac{1}{t}\log |\text{det} T^t_w|=\sum_{i=0}^{p+1} \text{dim} E_i(w) \eta_i(w):=\eta(w).$$   Here $E_0$ and $E_{p+1}$ could be zero.

The relation between Lyapunov and parallel Lyapunov exponents is; for stable $Z^s$ and unstable $Z^u$ vectors in $T_wH\Omega$,
\begin{eqnarray}\label{Lya}
\chi(Z^s)=-1+\eta(w,d\pi(Z^s)), \chi(Z^u)=1+\eta(w,d\pi(Z^u)),
\end{eqnarray} where $\pi:H\Omega\ra \Omega$ is the projection. See  Proposition 1 of \cite{cra} for details.
\begin{remark} The reversibility of the geodesic flow induces for regular points a special property observed by Crampon namely
$$\eta((x,[\psi]),v) = - \eta(\sigma(x,[\psi]), \sigma (v))$$
This means that if a trajectory of the flow is regular then the two opposite boundaries at $\infty$ are strongly related.
\end {remark}
Hence one can summerize these facts as:
\begin{eqnarray}\label{Eta}
TH\Omega=\mathbb R X \oplus (\oplus_{i=0}^{p+1}(E_i^s\oplus E_i^u))\end{eqnarray}
 where $E_i^s=J^X(E_i^u)$ for some pseudo-complex structure $J^X$ and $\chi_i^-$ is a Lyapunov exponent for $E_i^s$, $\chi_i^+$ for $E_i^u$, $E_i=E_i^s\oplus E_i^u$
$$\chi_i^+=1+\eta_i, \chi_i^-=-1+\eta_i, \chi_i^+=\chi_i^-+2,$$$$-2=\chi_0^-<\chi_1^-<\cdots<\chi_{p+1}^-=0=\chi_0^+<\chi_1^+<\cdots <\chi_{p+1}^+=2.$$
Note that 
$$\chi^+=\sum \text{dim} E_i^u \chi_i^+=\sum \text{dim} E_i^u(1+\eta_i)=\sum \text{dim} E_i^u +\sum \text{dim} E_i^u \eta_i$$
$$=(n-1)+\sum \text{dim} E_i^u \eta_i.$$
But
$$\eta=\sum \text{dim} E_i\eta_i=2 \sum  \text{dim} E_i^u \eta_i.$$
Hence
\begin{eqnarray}\label{relation}
\chi^+=(n-1)+\frac{1}{2}\eta.
\end{eqnarray}


The Lyapunov exponents have to do with the convexity of the boundary of $\Omega$. For 2-dimensional $\Omega$, the neighborhood of each point of the boundary $\partial \Omega$ can be written as the graph of a convex function $f$ near the origin. Such a function $f$ is said to be {\it approximately $\alpha$-regular} at the origin for an $\alpha\in [1,\infty]$, if
$$\lim_{t\ra 0} \frac{\log \frac{f(t)+f(-t)}{2}}{\log |t|}=\alpha.$$ This quantity is invariant under affine and projective
transformations. Approximately $\alpha$-regular for $\alpha$ finite means that the function behaves like $|t|^\alpha$ near the origin. 

Let $\Omega$ be a strictly convex open domain in $\R \P^2$ with $C^1$ boundary. Let $\xi\in\partial\Omega$. Let $w=(x,[\psi])\in H\Omega$ be such that  $\psi(\infty)=\xi$. Let $\cal H_w$ be a horocycle based at $\xi$ and passing through $x$.
It is shown in \cite{cra} (Theorem 1) (or \cite{Crampon} (Proposition 3.4.12)) that for any $v(w)\in T_x\cal H_w$,
\begin{eqnarray}\label{boundary}
\eta(w,v(w))=\frac{2}{\alpha(\xi)}-1.
\end{eqnarray}
\begin{remark} According to the remark about reversibility, for all regular trajectory $ (x, [\psi]) $ we have 
$$ \frac{1 }{ \alpha(\xi)} + \frac{1}{ \alpha(\sigma( \xi))}  =1 $$ where $\sigma(\xi)=\psi(-\infty)$.
\end{remark}

Periodic orbits are regular trajectories, and all these quantities can be expressed via the representation.
Suppose $\gamma$ is a hyperbolic isometry whose eigenvalues are $\lambda_1>\lambda_2>\lambda_3$ and $\gamma^+=\psi(\infty)$. Then it is shown in Proposition 5.5 of \cite{crampon} (or in Section 3.6 of \cite{Crampon}) that
\begin{eqnarray}\label{lambda}
\eta(w, v(w))=-1+2 \frac{\log\frac{\lambda_1}{\lambda_2}}{\log\frac{\lambda_1}{\lambda_3}}, \end{eqnarray}
 hence 
$$\alpha(\gamma^+)^{-1}=\frac{\log\frac{\lambda_1}{\lambda_2}}{\log\frac{\lambda_1}{\lambda_3}}.$$  


\subsubsection{Invariant measures}
For a geodesic flow on $H\Omega$, the maximal entropy $h_\mu(\phi)$ of the probability measure $\mu$ is known to be realized at the Bowen-Margulis measure.
This entropy is equal to the topological entropy $h_{top}(\phi)$  of the geodesic flow $\phi$ and it is also equal to the exponential growth of the lengths of  closed geodesics:
$$\lim_{R\ra\infty} \frac{\log \#\{[\gamma]|\ell(\gamma)\leq R\}}{R}.$$  This is again equal to the critical exponent of the associated Poincar\'e series. It is proved by Crampon \cite{crampon} that the entropy of Bowen-Margulis measure of strictly convex real projective structures  on a closed surface is between 0 and 1. 
But there exists another invariant measure called {\bf Sinai-Ruelle-Bowen measure}, abbreviated SRB measure. It is characterized as follows.
First we recall a fundamental theorem known as Margulis-Ruelle inequality:
\begin{theorem}Let $M=\Omega/\Gamma$ be a strictly convex compact projective manifold. Let $\mu$ be a geodesic flow invariant probability measure on $HM$. Then
$$h_\mu(\phi)\leq \int \chi^+ d\mu,$$ where $\chi^+=\sum \text{dim} E_i\cdot \chi_i^+$ denotes the sum of positive Lyapunov exponents.
\end{theorem}

An invariant measure that achieves equality in Margulis-Ruelle inequality is called  Sinai-Ruelle-Bowen measure.

Another description of this measure $\mu_{SRB}$ is: there exists a set $V$ of full Lebesgue measure
such that for each continuous function $f:M\ra \R$ and for every $x\in V$,
$$\lim_{T\ra\infty} \frac{1}{T}\int_0^T f(\phi^s(x))ds=\int f d\mu_{SRB}.$$

Ledrappier-Young \cite{LY} proved the following characterization of a Sinai-Ruelle-Bowen measure:
\begin{theorem}Let $M$ be a compact, strictly convex real projective manifold. Then the geodesic flow invariant measure $\mu$ is the SRB measure if and only if it has absolutely continuous conditional measures on unstable manifolds.
\end{theorem}

\subsection{Deformation of the representation}\label{Bergeron}
Ehresmann-Thurston described the deformation of geometric structures in terms of the deformation of
developing maps. This theory is rigorously rephrased by Bergeron-Gelander \cite{BG}.

\begin{proposition}
If $\rho_t:\pi_1(M)\ra SL(n,\mathbb R)$  is a smooth deformation of strictly convex real projective
structures, then there exists a continuous associated deformation $D_t$ of the developing map such that
 its image $\Omega_t$ and  the boundary $\partial\Omega_t$
in Hausdorff topology vary continuously, hence the Hilbert metric $d_t$ and the geodesic flow $\phi_t$ vary continuously. Furthermore the convexity forces also the continuity of $\partial \Omega_t$ in
$C^1$-topology.
\end{proposition} 
\begin{proof}For any two points $x,y\in M$, and two lifts $\tilde x, \tilde y\in \widetilde M$,
we know that $D_t \tilde x, D_t\tilde y$ vary continuously according to Bergeron-Gelander.  Then the line
connecting $D_t \tilde x, D_t\tilde y$ varies continuously. Now we need to determine two points on $\partial\Omega_t$
where this line intersects and show that they vary continuously. For each $\gamma\in \pi_1(M)$, let $\gamma_t^+$ be the 
line corresponding to the largest eigenvalue, $\gamma_t^-$ the line corresponding to the smallest eigenvalue, and $E_t^\gamma$ the sum of eigenspaces complementary to $\gamma_t^-$. Then it is known that $\gamma_t^+\in\partial \Omega_t$ and $T_{\gamma_t^+}\partial\Omega_t=E_t^\gamma$. Since $\rho_t$ varies continously, $\gamma_t^+$ and $E_t^\gamma$ vary
continuously.  For any given finite set $\gamma_1,\cdots,\gamma_k\in\pi_1(M)$,
$\Omega_t$ is included in the convex set formed by $\cap HE_t^{\gamma_i}$ where $HE_t^{\gamma_i}$ is the half space containing $\Omega_t$. By enlarging the set,
we can see that $\Omega_t$ and $\partial \Omega_t$ vary continuously.
\end{proof}

Note that each point in $\cal P_n(M)$ is a class of $(D,\rho)$ where $D:\widetilde M\ra \mathbb{RP}^n$
is a  developing map and $\rho:\pi_1(M)\ra PGL(n+1,\mathbb R)$ is a holonomy representation.
To each such pair $(D,\rho)$ are associated $\Omega_D=D(\widetilde M)$, the Hilbert metric and the geodesic flow, hence $h^D_{SRB}$ the SRB entropy.
If two structures $(D,\rho)$ and $(D',\rho')$ are equivalent, then $\rho$ and $\rho'$ are conjugate
and $\Omega_D$ and $\Omega_{D'}$ are diffeomorphic by a projective map. Since the entropy is
invariant under conjugacy, we have a well-defined map
$h_{SRB}:\cal P_n(M)\ra \mathbb R$.  It is known that $\cal P_n(M)$ is a component of a character
variety by Benoist \cite{Be}. Hence it is enough to study the continuity along a continuous path.

\section{Continuity of the Sinai-Ruelle-Bowen measure entropy and application}
Throughout this section $M$ is a compact smooth manifold, $D:\widetilde M\ra \R\P^n$ is a developing map, and
$\rho:\pi_1(M)\ra PSL(n+1,\R)$ is the holonomy map of a strictly convex real projective structure. Let $g^h$ be the Hilbert metric of $\Omega=D(\widetilde M)$,  $F$ the Hilbert norm on $TM$, and $\phi^t$ the geodesic flow associated with the Hilbert metric.  Then it is known that this flow is Anosov and it is $C^{1,\alpha}$ due to Benoist \cite{Be}.

Let $g$ be a $C^\infty$-Riemannian metric on $M$. Suppose $\psi^t$ is a flow with the same trajectory as $\phi^t$ but unit speed with respect to $g$. If
$$X=\frac{d\phi^t}{dt}, Z=\frac{d\psi^t}{dt},$$ then
$$X=mZ.$$  We know that
$$F(d\pi X)=g(d\pi Z)=1$$ where $\pi:HM\ra M$ is a projection. Then
$$1=m(x,[v])F(d\pi(Z(x,[v])), $$ and set
\begin{eqnarray}\label{C^1}
\beta=m^{-1}=F(d\pi Z)
\end{eqnarray}

\begin{lemma}$\psi^t$ and $Z$ are $C^\infty$ and $\phi^t$ and $X$ are $C^{1,\alpha}$. Furthermore $\beta:HM\ra\R^+$ is $C^{1,\alpha}$.
\end{lemma}
\begin{proof}The geodesic foliation of a strictly convex projective structures is smooth because locally it is the foliation by straight lines in a projective chart. This was previously observed by  Benoist. The fact that $\phi^t$ is $C^{1,\alpha}$ is also due to Benoist.
The $C^1$-ness comes from the $C^1$-ness of Finsler norm, which is again due to the $C^1$-ness of $\partial \Omega$. The generator of the geodesic flow, $X$ is also
$C^{1,\alpha}$, being tangent to a smooth foliation and being  normalized with respect to the Hilbert metric.
But if we reparameterize the flow, using a smooth Riemannian metric, the resulting vector field $Z$ is smooth,  so it is  for the induced flow.
However  $F$ is $C^{1,\alpha}$, hence $\beta$ is $C^{1,\alpha}$.
\end{proof}

We have the following theorem about the change of time, due to Anosov and Sinai \cite{Sinai, Anosov}.
\begin{theorem}If $\phi$ is a $C^1$-Anosov flow and $\psi$ is obtained from $\phi$ by multiplying a positive $C^1$-function on the speed, then $\psi$ is again Anosov. 
\end{theorem}
\begin{remark}In our case, $\psi$ is a $C^\infty$-Anosov flow. This is the $C^\infty$ foliation by straight lines with  $C^\infty$ parametrization.
\end{remark}

Now using the result of Parry \cite{Anosov} we get,
\begin{eqnarray}
\mu^\phi_{SRB}=\beta\mu^\psi_{SRB}/\int_{HM}\beta d\mu^\psi_{SRB}
\end{eqnarray}
and using a theorem of Abramov \cite{Abramov}
\begin{eqnarray}\label{C^r}
h_{SRB}(\phi)=h_{SRB}(\psi)/\int_{HM}\beta d\mu^\psi_{SRB}.
\end{eqnarray}

The continuity of the SRB entropy for $C^\infty$ Anosov flow is due to Contreras \cite{Contreras} (theorem B).
\begin{theorem}\label{C^{r-2}}If $\lambda$ is a $C^r$-Anosov flow, there exists a neighborhood $U$ of $\lambda$ in $C^r$ topology such that the function
$\psi\ra h_{SRB}(\psi)$  is  $ C^{r-2}$.\\
\end{theorem}

To control the denominator let us recall that the SRB measure is the unique equilibrium state associated with the infinitesimal volume expansion which is here a smooth function. Since the scaling factor $\beta$  is $C^{1,\alpha}$,  we may again invoke \cite{Contreras} (theorem C) that we may specialize to SRB measures as follows

\begin{theorem}\label{C^{r-1}}If $\lambda$ is a $C^r$-Anosov flow, there exists a neighborhood $U$ of $\lambda$ in $C^r$ topology such that if  $\mu_{\lambda}$ is its SRB measure then for any $0< \alpha <1$ small enough the map $ \lambda \in U \rightarrow  \mu_{\lambda} \in C^{\alpha}(M, \R)^*$ is $C^{r-1}$.
\end{theorem}

%

Now we prove that the entropy of SRB along a smooth deformation varies continuously.
\begin{theorem}Let $\rho_t:\pi_1(M)\ra SL(n+1,\R)$ be a smooth deformation of strictly convex projective structures.
Then $t\ra h_{SRB}(\rho_t)$ is a continuous map.
\end{theorem}
\begin{proof}Since $\Omega_t=D_t(\widetilde M)$ is continuously determined by $\rho_t$ as in Section {\ref{Bergeron}}, both
$\Omega_t$ and $\partial\Omega_t$ depend continuously on $t$ with respect to the Hausdorff topology. Fix a $C^\infty$ Riemannian metric on $M$. Now since $\rho_t$ depends smoothly on $t$, the reparameterized geodesic flow $\psi_t$ depends smoothly on $t$. By Equation (\ref{C^1}), $$\beta_t=F_t(d\pi Z_t)$$ is continuous in $t$.  Then, by Equation (\ref{C^r}), and since $h(\psi_t)_{SRB}$ varies continuously by Theorem \ref{C^{r-2}}, it is enough to show that the denominator varies continuously.
\begin{lemma}Let $\alpha$ be as in theorem \ref{C^{r-1}}. If $f_i  \in C^{\alpha}(M,\R) $  converges to $f \in C^{\alpha}(M,\R) $in $C^0$ topology and probability measures $\mu_i$  converge to $\mu$ in $C^{\alpha}(M, \R)^*$, then 
$$\int_M f_i d\mu_i\ra \int_M f d\mu.$$
\end{lemma}
\begin{proof} $$| \int f_i d\mu_i - \int f d\mu|\leq  |\int f_i d\mu_i -\int f d\mu_i|+|\int f d\mu_i -\int f d\mu|$$$$\leq |f_i-f|_\infty + |\int f d(\mu_i-\mu)|.$$
The first term goes to zero since $f_i\ra f$ in $C^0$ topology and the second term goes to zero
since $\mu_i\ra\mu$ in $C^{\alpha}(M, \R)^*$.
\end{proof}

%
\end{proof}

\begin{corollary}Along the smooth deformation of convex real projective structures, the sum $\eta$ of parallel Lyapunov exponent varies continuously.
\end{corollary}
\begin{proof}Since the geodesic flow $\phi$ is ergodic with respect to the SRB probability measure $\mu$, and the measurable function $\chi^+$ is $\phi$-invariant, $\chi^+$ is almost constant. Furthermore, by the property of the SRB measure, $h_\mu(\phi)=\int \chi^+ d\mu$, the entropy $h_\mu(\phi)$ must be equal to $\chi^+$.
Since the entropy varies continuously along the smooth deformation, $\chi^+$ varies continuously along the smooth deformation. The claim follows from Equation (\ref{Lya}).
\end{proof}
In particular, in the case of dimension 2 surfaces, it can be stated as 
\begin{corollary} Along the smooth deformation of the convex structure, $\alpha$ varies continuously.
\end{corollary}

Now we give a proof for Corollary \ref{irreversible} 
\begin{proof}The continuity follows from the main Theorem \ref{main} and the other claims follow from results in \cite{Crampon} and equalities $\chi_i^+=1+\eta_i, \chi_i^-=-1+\eta_i$, $\chi^+\alpha=2$. (4) Let $\omega$ be a volume form. Note that $THM=E^s\oplus E^u \oplus \mathbb R X$ where $X$ is the vector field generating the geodesic flow $\phi^s$ of the Hilbert metric. Choose a Riemannian metric $g$ so that for each
point $w\in T_wHM$, $e_1(w),\cdots,e_{n-1}(w)$ are bases of $E^u(w)$ consisting of unit vectors, and
$b_1(w),\cdots,b_{n-1}(w)$ bases of $E^s(w)$ consisting of unit vectors. Since $M$ is compact and $\omega$ is a continuous volume form, $$\omega(e_1,\cdots,e_{n-1},b_1,\cdots,b_{n-1}, X)<N$$
on $HM$ for some positive $N$.
Since the flow $\phi^s$ preserves the splitting of $E^u, E^s$,
$$d\phi^s(e_i (w))\in E^u(\phi^s(w)), d\phi^s(b_i(w))\in E^s(\phi^s(w)), d\phi^s(X)=X.$$
Then
$$\lim_{s\ra\infty}\frac{1}{s}\log(\frac{d(\phi^s)^*(\omega)(e_i,b_i,X)}{\omega(e_i,b_i,X)})=$$
$$\lim_{s\ra\infty}\frac{1}{s}\log(\frac{\omega(d\phi^s(e_i,b_i,X))}{\omega(e_i,b_i,X)})=\lim_{s\ra\infty}\frac{1}{s}\log(\omega(d\phi^s(e_i,b_i,X)))=$$
$$\lim_{s\ra\infty}\frac{1}{s}\log(|\text{det} d\phi^s|)=\sum \chi_i \text{dim} E_i= \eta. $$
The last equality follows from the fact that $\chi_i^+=1+\eta_i, \chi_i^-=-1+\eta_i$ in Equation (\ref{Eta}) and the second to  last equality follows from Equation (\ref{det}).
\end{proof}
\begin{proof} For Corollary  \ref{shape} and geometric discussion.
By ergodicity of the SRB measure, there exists a flow invariant  set $ W_{SRB} \subset HM$ of full measure with respect to the SRB measure on which the sum $\chi^+$ of positive Liapunov exponents is  constant and coincide with the value of the SRB entropy.
Furthermore, this shows that the sum of exponents $\eta$ for the parallel transport is, on that set, a non positive number $ \eta \leq 0$ such that $h_{SRB}= n-1 +   \frac{1}{2}\eta$ according to Eqn (\ref{relation}) and Crampon's inequality for any invariant measure.
We observe that $ \eta = 0$ is equivalent  to being  Riemannian hyperbolic case. In the hyperbolic case every orbit of the flow is regular. We will observe that this is no longer true when $ \eta < 0$, which is what we will assume in the sequel.
The set $W_{SRB} \subset HM$ being of full SRB measure, for almost all unstable leaf with respect to the invariant transverse measure induced by the SRB measure, almost every trajectory with respect to the conditional measure  on that unstable leaf is regular and  the sum of exponents for the parallel transport is $\eta$. We know that the conditional measure on unstable leaves of the SRB measure are absolutely continuous with respect to the  Lebesgue  measure therefore for each such unstable leaf $L^u (x ,[v])$ there exists a subset  of full Lebesgue measure  $S_{L^u (x ,[v])} \subset \partial \Omega$ such that for any boundary point $\xi \in S_{L^u (x ,[v])}$ is approximately regular with exponent 
\begin{eqnarray}\label{boundary exponent}
  \alpha(\xi)= \frac{2}{1+ \frac{\eta}{2}} \geq  2
\end{eqnarray}
Again, the equality is achieved only in the hyperbolic case.
We can go further and observe some strange effect that will shed some light on the transversal measure of the SRB measure. Such an unstable leaf is associated to a point $\xi ^-$ on the boundary  such that all the geodesics under consideration emanate from that point at $-\infty$.
However we know that any regular geodesic is also also regular by reversing time and then all the reversed geodesics have opposite parallel transport exponent. Then the approximately regular  exponent  at $\xi ^-$ is  $\alpha(\xi^-)= \frac{2}{1- \frac{\eta}{2}} \leq 2$.
Such points have a measure 1 with respect to the tranverse SRB measure which is therefore orthogonal to Lebesgue as soon as it is not hyperbolic. To sum up, there exists on the boundary a set of fulll Lebesgue measure of  points on which the convex is ``flatter" than a conic and a zero Lebesgue measure of points where the convex is ``sharper" than a conic.
A funny remark is that when $ \eta <0 $ there is no triangle with regular geodesics but there are many even sided regular polygons with that regularity.
\end{proof}

\vskip .1 in
\noindent     Patrick Foulon\\ Aix-Marseille Universit\'e, CNRS, Soci\'et\'e Math\'ematique de France, CIRM (Centre International de Rencontres Math\'ematiques), Marseille, France.
\\UMR 822, 163 avenue de Luminy
\\13288 Marseille cedex 9, France\\
\texttt{foulon\char`\@cirm-math.fr}\\
\vskip .005 in
\noindent     Inkang Kim\\
     School of Mathematics\\
     KIAS, Heogiro 85, Dongdaemen-gu\\
     Seoul, 02455, Korea\\
     \texttt{inkang\char`\@ kias.re.kr}

\begin{thebibliography}{99}
\bibitem{Abramov}L. Abramov, On the entropy of a flow, Am. Math. Soc. Transl. 49 (1966), 167-170.
\bibitem{Sinai} D. Anosov and Y. Sinai, Some smooth ergodic systems, Russ. Math. Survey. 22 (1967), 103-167.
\bibitem{Be}Y. Benoist, Convexes divisibles I, Algebraic groups and arithmetic, Tata Inst. Fund. Res. Stud. Math, 17 (2004), 339-374.
\bibitem{BH}Y. Benoist and D. Hulin, Cubic differentials and finite
volume convex projective surfaces, JGD, 98 (2014), 1-19.
\bibitem{Ben}J-P Benz\'ecri, Sur les vari\'et\'es localement affines
et localement projectives, Bull. Soc. Math. France, 88 (1960),
229-332.
\bibitem{BG}N. Bergeron and T. Gelander, A note on local rigidity, Geom. Dedicata, 107 (2004), 111-131.
\bibitem{bg}F. Bonahon and G. Dreyer, Parameterizing Hitchin components, Duke Math. J. 163 (2014), no. 15, 2935--2975.
\bibitem{bk}F. Bonahon and I. Kim, The Goldman and Fock-Goncharov coordinates for
convex projective structures on surfaces, Geom. Dedicata, 192 (2018), 43–-55.
\bibitem{Bowen} R. Bowen, Periodic orbits of hyperbolic flows, Amer. J. Math., 94 (1972).
\bibitem{BL}O. Butterley and C. Liverani, Smooth Anosov flows: Correlation spectra and stability, J. Mod. Dyn. 1 (2007), 301-322.
\bibitem{Cal}E. Calabi, Complete affine hyperspheres I., In Symposia
Mathematica, Vol. X, pages 19-38, 1972.
\bibitem{CY}S. Cheng and S. Yau, On the regularity of the
monge-amp\'ere equation $det(\partial^2 u/\partial x_i\partial s
x_u)=F(x,u)$, Comm. Pure Appl. Math., 30 (1977), 41-68.
\bibitem{cg}S. Choi and W. Goldman, Convex real projective structures on closed surfaces are closed, Proc. Amer. Math. Soc. 118 (1993), no. 2, 657-–661.
\bibitem{Contreras}G. Contreras, Regularity of topological and metric entropy of hyperbolic flows, Math. Z, 210 (1992), 97-111,
\bibitem{Crampon}M. Crampon, Dynamics and entropies of Hilbert metrics, Th\`ese, Universit\'e de Strasbourg, 2011.
\bibitem{crampon}M. Crampon, Entropies of strictly convex projective
manifolds, J. Mod. Dyn. 3 (2009), 511-547.
\bibitem{cra}M. Crampon, Lyapunov exponents in HIlbert geometry, Ergodic theory Dynam. System 34 (2014), 501-533.
\bibitem{cra1}M. Crampon, The boundary of a divisible convex set, Publ. Mat. Urug. 14 (2013), 105-119.
\bibitem{fg}V. Fock and A. Goncharov, Moduli spaces of convex projective structures on surfaces, Adv. Math. 208 (2007), no. 1, 249–-273. 
\bibitem{Fu1}P. Foulon, G\'eom\'etrie des \'equations diff\'erentielle du second ordre, Ann. Inst. Henri Poincar\'e, 45 (1986), 1-28.
    \bibitem{Fu2}P. Foulon, Estimation de l'entropie des syst\`emes lagrangiens sans points conjugu\'es, Ann. Inst. Henri Poincar\'e Phys. Th\'eor., 57 (2) (1992), 117-146. With an appendix, ``About Finsler geometry''.
\bibitem{Goldman}W. Goldman, Convex real projective structures on
compact surfaces, JDG, 31 (1990), 791-845.
\bibitem{Hi}N. Hitchin, Lie groups and {T}eichm\"uller space,
{Topology},
\textbf{31}(1992),
 {449--473. }
\bibitem{KZ}I. Kim and G. Zhang, K\"ahler metric on the space of convex real projective structures on surface, JDG, 106 (2017), 127-137.
\bibitem{koba}S. Kobayashi, Foundations of Differential Geometry,
Volume 1, Interscience Publishers, 1963.
\bibitem{Lab}F. Labourie, Flat projective structures on surfaces and cubic holomorphic differentials, Pure Appl. Math. Q, 3 (2007), 1057-1099.
\bibitem{La}F. Labourie, Anosov flows, surface groups and curves in projective space, Invent. Math. 165 (2006), 51-114.
\bibitem{La1}F. Ledrappier, Structure au bord des vari\'et\'es \`a courbure n\'egative, S\'eminaire de th\'eorie spectrale et g\'eom\'etrie, 71, 1994-1995.
\bibitem{LY}F. Ledrappier and L.-S. Young, The metric entropy of diffeomorphism, Ann. of Math., 122 (1985), 509-574.
\bibitem{Lo2}J. Loftin, Affine spheres and convex $RP^n$ manifolds,
American Journal of Math. 123 (2) (2001), 255-274.
\bibitem{Lo}J. Loftin, The compactification of the Moduli space of
convex $RP^2$-surfaces I, JDG 68 (2004), 223-276.


\bibitem{Os}V.I. Oseledec, A multiplicative ergodic theorem, Trans. Moscow Math. Soc., 19 (1968), 197-231.
\bibitem{Anosov}W. Parry, Synchronisation of Canonical measures for hyperbolic attractors, Comm. Math. Phys. 106 (1986), 267-275.
\bibitem{Ru} D. Ruelle, Differentiation of SRB states for hyperbolic flow, Ergodic Theory Dynam. Systems 28 (2008), 613-631.
\end{thebibliography}
\end{document}